\documentclass[11pt,a4paper,reqno]{article}
\usepackage{stix}
\usepackage{amsmath,amsfonts}
\usepackage{amsthm} 
\usepackage{upref}
\usepackage{enumerate}
\usepackage{verbatim}
\usepackage{tikz}
\usepackage{authblk}
\usepackage[margin=1in]{geometry}
\usepackage{tikz}
\usepackage[colorlinks=true, citecolor={blue}]{hyperref}
\usepackage{mathrsfs}  

\title{Non-dissipative system as limit of a dissipative one: \\ comparison of the asymptotic regimes \\ \ \\
Syst\`eme non dissipatif comme limite d'un syst\`eme dissipatif: \\ comparaison des r\'egimes asymptotiques}

\author{Ricardo Parreira da Silva}
\affil{Department of Mathematics, University of Bras\'ilia, \\
Campus Universit\'ario Darcy Ribeiro, Bras\'ilia, DF, 70910-900, Brazil
\authorcr%
\nolinkurl{rpsilva@unb.br}}%

\date{}

\makeatletter
\hypersetup{pdfauthor={Ricardo Parreira da Silva}}
\hypersetup{pdftitle={\@title}}
\makeatother

\newtheorem{theorem}{Theorem}[section]
\newtheorem{lemma}[theorem]{Lemma}
\newtheorem{proposition}[theorem]{Proposition}
\newtheorem{corollary}[theorem]{Corollary}

\theoremstyle{definition}
\newtheorem{definition}[theorem]{Definition}

\theoremstyle{remark}
\newtheorem{remark}[theorem]{Remark}

\numberwithin{equation}{section}

\let\oldthebibliography\thebibliography
\renewcommand\thebibliography[1]{
	\oldthebibliography{#1}
	\setlength{\parskip}{0pt}
	\setlength{\itemsep}{0pt plus 0.3ex}
	\small}

\newenvironment{poliabstract}[1]
   {\renewcommand{\abstractname}{#1}\begin{abstract}}
   {\end{abstract}}

\begin{document}

\maketitle
\begin{poliabstract}{Abstract} 
 Let $\Omega \subset \mathbb{R}^n$ be a bounded smooth domain (open and connected) in $\mathbb{R}^n$.		
Given $u_0\in L^2(\Omega)$, $g\in L^\infty(\Omega)$ and $\lambda \in \mathbb{R}$, our purpose is to describe the asymptotic behavior of solutions of the family of problems
\begin{equation*}
\left\{	
\begin{array}{rcll}
\dfrac{\partial u}{\partial t} - \Delta_p u & = & \lambda u + g, &  \text{ on } \quad (0,\infty)\times \Omega, \\
u & = & 0, & \text{ in } \quad (0,\infty)\times \partial \Omega, \\
u(0, \cdot) & = & u_0, & \text{ on } \quad\Omega,		
\end{array}
\right.	
\end{equation*}
as $p \longrightarrow 2^+$, where $\Delta_p u:=\rm{div}\big(|\nabla u|^{p-2}\nabla u\big)$ denotes the $p$-laplacian operator.

\end{poliabstract}
 
\begin{poliabstract}{R\'esum\'e}
Soit $\Omega \subset \mathbb{R}^n$ un domaine lisse d\'elimit\'e (ouvert et connect\'e) dans $\mathbb{R}^n$.
\'Etant donn\'e $u_0\in L^2(\Omega)$, $g\in L^\infty(\Omega)$ et $\lambda \in \mathbb{R}$, notre but est de d\'ecrire le comportement asymptotique de solutions de la famille de probl\`emes
\begin{equation*}
\left\{	
\begin{array}{rcll}
\dfrac{\partial u}{\partial t} - \Delta_p u & = & \lambda u + g, &  \text{ en } \quad (0,\infty)\times \Omega, \\
u & = & 0, & \text{ en } \quad (0,\infty)\times \partial \Omega, \\
u(0, \cdot) & = & u_0, & \text{ en } \quad\Omega,		
\end{array}
\right.	
\end{equation*}
comme $p \longrightarrow 2^+$, ou $\Delta_p u:=\rm{div}\big(|\nabla u|^{p-2}\nabla u\big)$ d\'esigne l\'e op\'erateur $p$-laplacian.
\end{poliabstract}


{ 
\renewcommand{\thefootnote}{}%
	}
\section{Introduction}\label{sec:intro}

Let $\Omega \subset \mathbb{R}^n$ be a bounded smooth domain in $\mathbb{R}^n$, $n \geq 1$. 
Given $u_0\in L^2(\Omega)$, $g\in L^\infty(\Omega)$ and $\lambda \in \mathbb{R}$, consider the family of problems parametrised by $p \in [2,\infty)$
\begin{equation}\tag{$E_p$}\label{eq:Ep}
\left\{	
\begin{array}{rcll}
\dfrac{\partial u}{\partial t} - \Delta_p u & = & \lambda u +g, &  \text{ on } \quad (0,\infty)\times \Omega, \\
u & = & 0, & \text{ in } \quad (0,\infty)\times \partial \Omega, \\
u(0, \cdot) & = & u_0, & \text{ on } \quad\Omega,		
\end{array}
\right.	
\end{equation} 
where $\Delta_p u := \rm{div}\big(|\nabla u|^{p-2}\nabla u\big)$ denotes the $p$-laplacian operator.

As pointed out by the authors in \cite{CCD:99}, if $p>2$ the nonlinear laplacian give rise to a strong dissipation mechanism which ensures the existence of a compact global attractor (see \cite{BV:92, Ha:88}) for the semiflow generated by \eqref{eq:Ep} in the state space $L^2(\Omega)$, unlike from the case $p=2$. 

If $p=2$, equation \eqref{eq:Ep} reads
\begin{equation}\tag{$E_2$}\label{eq:E2}
\left\{	
\begin{array}{rcll}
\dfrac{\partial u}{\partial t} - \Delta u & = & \lambda u + g, & \text{ on } \quad (0,\infty)\times \Omega, \\
u & = & 0, & \text{ in } \quad (0,\infty)\times \partial \Omega, \\
u(0, \cdot) & = & u_0, & \text{ on } \quad \Omega,		
\end{array}
\right.	
\end{equation}
and is not difficult to show the existence of unbounded orbits for the semiflow generated by \eqref{eq:E2} if $\lambda$ is large enough (see Section \ref{sec:p2}). 
Such behavior is typical of systems referred to in the recent literature as slowly non-dissipative dynamical systems (see \cite{BCP:18, CP:18, PR:16}). 

We must recall that a compact global attractor is a nonempty maximal compact invariant set attracting each bounded subset in the state space. 
In the context of slowly non-dissipative systems, which are characterised by the existence of unbounded orbits with the absence of finite time blowup, one cannot expect compactness for a proper extension of the notion of global attractor.
In such case, the object to be considered would be nonempty, maybe unbounded, but minimal in the category of invariant sets which attracts all bounded subsets in the state space. 
This is referred to as non-compact global attractor (see Definition \ref{def:unbattr}) and \cite{BCP:18, CP:18, PR:16}).

Our aim in this note is to address to the stability with respect to the parameter $p$ of such objets, comparing compact attractors in the dissipative case $p>2$ with the non-compact attractors in the non-dissipative limiting case $p=2$ in terms of the Hausdorff semi-distance between then (see Section \ref{sec:cont-attr}).

Despite of the intense development of the theory of dissipative systems, the called non-dissipative systems still needs a deeper understanding and we believe that studies as we carrie out on this paper may be of great motivation for that.

This paper is structured as follows. 
In Section \ref{sec:pnot2} we address to the dissipative case. 
In particular we setup the functional framework and present preliminary results concerned with the asymptotic behavior of solutions of \eqref{eq:Ep}. 
In Section \ref{sec:p2} we address to the non-dissipative case. 
We show the existence of unbounded orbits and their behaviour at infinity. 
We consider the notion of non-compact global attractor and by considering Poincar\'e's compactification as in \cite{He:13}, it is possible to interpret such attractors as heteroclinics orbits connecting equilibrium points at infinite.

In Section \ref{sec:cont} we address to the continuity with respect to the parameter $p$ of the orbits as well of the compact and non-compact global attractors.
\section{The dissipative case: $p>2$}\label{sec:pnot2} 

Let $H$ be the Hilbert space $L^2(\Omega)$ endowed with the standard inner product and $V$ the reflexive space $W_0^{1,p}(\Omega)$ endowed, thanks to the Poincar\'e inequality, with the equivalent norm $\displaystyle \|u\|_V= \left[\int_\Omega |\nabla u(x)|^p dx\right]^\frac1p$. 
Identifying $H$ with $H^*$, its topological dual space, we can write
$$
V \stackrel{d} \hookrightarrow  H \stackrel{d} \hookrightarrow  V^*,
$$
i.e., each space is dense in the following one and the inclusion maps are continuous. 
In this particular case, $V^*=W^{-1,q}(\Omega)$, where $q$ is the conjugate exponent of $p$, $1/p + 1/q =1$.

The inner product in $H$ will be denoted by $\langle \cdot,\cdot \rangle$ whereas $\langle \cdot, \cdot \rangle_{V^*,V}$ will denote the duality product between $V^*$ and $V$.
As a consequence of the previous identification
$$
\langle u,v \rangle  = \langle u, v \rangle_{V^*,V}, \quad \text{for all} \;  u \in H \;  \text{and} \; \text{for all} \; v \in V.
$$
Let us now consider the (nonlinear) operator $ A_p : V \to V^*$ defined by
$$
\langle A_p u, v \rangle_{V^*,V} := \int_{\Omega} |\nabla u|^{p-2} \nabla u \cdot \nabla v  \, dx, \quad \forall \, v \in V,
$$
where for $\xi, \eta \in \mathbb{R}^n$, $\xi \cdot \eta$ denotes the standard inner product in $\mathbb{R}^n$.

It is well known (see \cite{Lions:69, Te:88}) that operator $A_p$ satisfies the following structural properties:
\begin{enumerate}
	\item [$i)$] The operator $A_p : V \to V^*$ is maximal monotone, coercive and hemicontinuous;
	
	\item [$ii)$] There are constants $c_1, \, c_2 >0$ such that for all $v \in V$ the following conditions hold:
	$$
	\langle A_p v, v \rangle_{V^*,V} \geq c_1 \|v\|^p_V \quad \text{ and } \quad \|A_p v \|_{V^*} \leq c_2 \|v\|^{p-1}_V.
	$$ 
\end{enumerate}
This allow us, in particular, to conclude that the set $D(A_p):=\{u \in V: A_p u \in H\}$ is dense in $H$ (see \cite[Lemma 1]{CCD:99}).

Considering the $H$-realization of the operator $A_p$, i.e. the restriction of the operator $A_p$ on $D(A_p)$ seen as an operator on $H$ (which we will still denote by $A_p$), it is well known (see \cite[Example 2.3.7]{Br:73}) that $A_p$ is a maximal monotone operator in $H$. 
Therefore, $A_p$ can also be seen as the subdifferential of the lower semicontinuous convex functional $J_p: H \to (-\infty, \infty]$ defined by
$$
J_p(u):=\begin{cases}
	\dfrac1p\,  \|u\|_V^p, & \text{ if } u \in V, \\
	\infty,   & \text{ otherwise}.
\end{cases}
$$

The problem \eqref{eq:Ep} can be now be written in the form
\begin{equation}\label{eq:AEp}	
\begin{array}{rcl}
\dfrac{du}{dt} + A_p u  & = & B(u), \quad t>0, \\
u(0) & = & u_0 \in H,  	
\end{array}
\end{equation} 
where $B:H \to H$ is the globally Lipschitz map defined as $B(u):=\lambda u +g$.

\begin{definition}
A function $u \in \mathcal{C}([0,\infty); H)$ is a strong solution to \eqref{eq:AEp} starting at $u_0\in H$ if $u$ is absolutely continuous in any compact subinterval of $(0,\infty)$, $u(t) \in D(A_p)$ for a.e. $t \in (0,\infty)$, $u(0)=u_0$ and
\begin{equation}\label{eq:s1}
\dfrac{du}{dt}(t) + A_p u(t) = B(u(t)), \quad \text{ for a.e. } \, t \in (0,\infty).
\end{equation}
\end{definition}

\begin{theorem}[\cite{KS:14}, Theorem 2.8]\label{teo:KS}
Let $p > 2$ be fixed. For all $u_0 \in H$ there exist a unique strong solution $u_p \in \mathcal{C}([0,\infty); H)$ to \eqref{eq:AEp} starting at $u_0$.
\end{theorem}

Theorem \ref{teo:KS} allow us to consider, for each $p>2$, the (nonlinear) semigroup $\{T_p(t)\}_{t\geq 0} \subset \mathcal{C}(H;H)$ 
$$
T_p(t):u_0 \in  H \mapsto u_p(t) \in H, 
$$
the strong solution to \eqref{eq:AEp} starting at $u_0$ at the time $t$.

\begin{definition}[\cite{Ha:88}]
An subset $\mathcal{A}_p \subset H$ is said to be a compact global attractor for the semigroup $\{T_p(t)\}_{t \geq 0}$ if $\mathcal{A}_p$ is a compact subset of $H$, $T_p(t)\mathcal{A}_p = \mathcal{A}_p$ for all $t\geq 0$ and $\displaystyle \lim_{t\to \infty} \, \sup_{b \in \mathcal{B}} \inf_{a\in \mathcal{A}_p} \|T_p(t)b - a\|_{H}=0$, for all bounded subset $\mathcal{B} \subset H$. 
\end{definition}

\begin{definition}[\cite{Ha:88}]
A complete orbit for the semigroup 	$\{T_p(t)\}_{t \geq 0}$ through $u_0 \in H$ is a continuous function $\phi:\mathbb{R} \to H$ satisfying $\phi(0)=u_0$ and $T(t)\phi(s)=\phi(t+s)$, for all $s \in \mathbb{R}$ and $t \geq 0$.
\end{definition}

\begin{theorem}[\cite{Ha:88}]\label{teo:carac-att}
If $\mathcal{A}_p \subset H$ is the compact global attractor for the semigroup $\{T_p(t)\}_{t \geq 0}$, then
$$
\mathcal{A}_p = \{u_0 \in H: \text{ there exists a bounded complete orbit for }	 \{T_p(t)\}_{t \geq 0} \text{ through } u_0\}.
$$
\end{theorem}

\begin{theorem}\label{teo:CCD}
Let $p > 2$ be fixed. The semigroup $\{T_p(t)\}_{t \geq 0}$ is $(H,V)$-dissipative, i.e., there exists $R=R(p)>0$ such that for all bounded subset $\mathcal{B} \subset H$, there exists $T=T(\mathcal{B})>0$ such that $\displaystyle \sup_{\underset{v \in \mathcal{B}}{t\geq T}} \|T_p(t)v\|_V <R$.
In particular, $\{T_p(t)\}_{t \geq 0}$ has a compact global attractor $\mathcal{A}_p$. Moreover, for any $q \in (p,\infty)$, $R$ can be chosen uniformly for $2<p\le q$.
\end{theorem}
\begin{proof}
For the proof of existence of a compact global attractor in $H$ in a more general setting see \cite[Theorem 1]{CCD:99}. For the a priori estimate in $V$ see \cite[\textsection 5.2]{Te:88}.

\end{proof}

The next Proposition shows how simple the dynamics of the system \eqref{eq:AEp} is in the case $\lambda \leq 0$.

\begin{proposition}\label{prop:atratriv}
If $\lambda \leq 0$ then $\mathcal{A}_p=\{u^*_p\}$, where $u^*_p$ is the (unique) equilibrium point of \eqref{eq:Ep}, i.e., $u_p^*\in D(A_p)$ and $A_p u_p^* = \lambda u_p^* + g$. 
\end{proposition}
\begin{proof}
Setting $w(t) := u_p(t) - u_p^*$, one immediately obtains that 
\begin{equation}\label{eq:2.3}
\dfrac{dw}{dt}=A_p u_p^* - A_p u_p + \lambda w .	
\end{equation} 
Moreover, it follows from Tartar's inequality \eqref{lem:tartar} that
$$
\langle A_p u_p^* - A_p u_p, w \rangle_{V^*,V}  = \langle  |\nabla u_p^*|^{p-2} \nabla u_p^* - |\nabla u_p|^{p-2} \nabla u_p, \nabla w \rangle \leq -2^{2-p} \|\nabla w\|_{H}^p \le  -2^{2-p}\lambda_1^{p/2} \|w\|_{H}^p, 
$$
where $0<\lambda_1$ denotes the first eigenvalue of the Dirichlet Laplacian. 

Taking in \eqref{eq:2.3} the inner product in $H$ with $w$ one obtains that
$$
\dfrac12\dfrac{d}{dt}\|w \|_{H}^2 \leq \lambda \| w\|_{H}^2, \quad t>0, \qquad \text{ if } \lambda \neq 0
$$
and
$$
\dfrac12\dfrac{d}{dt}\|w \|_{H}^2 + 2^{2-p}\lambda_1^{p/2} (\|w\|_{H}^2)^{p/2} \leq 0,   \quad t>0, \qquad \text{ if } \lambda = 0.
$$ 

Therefore,
 $\|w(t)\|_H^2\le \|u_0 - u^*_p\|^2_H \, e^{2\lambda t} \stackrel{t \to \infty}{\longrightarrow} 0$, if $\lambda \neq 0$, and thanks to Ghidaglia's Inequality \ref{lem:ghidaglia}, $\|w(t)\|_H^2\le  \left(2^{2-p}\lambda_1^{p/2}(p-2)t \right)^{-2/(p-2)} \stackrel{t \to \infty}{\longrightarrow} 0$, if $\lambda =0$.
	
\end{proof}

In virtue of Proposition \ref{prop:atratriv} a more complicated dynamics for system \eqref{eq:Ep}, $p>2$, is expected only for $\lambda >0$.

\section{The non-dissipative case: $p=2$}\label{sec:p2} 

Let $A_2:D(A_2) \subset H \to H$ be the Laplace operator with homogeneous Dirichlet boundary conditions, i.e., $A_2:=-\Delta$ with domain $D(A_2)= W_0^{1,2}(\Omega)\cap W^{2,2}(\Omega)$. It is well know (see \cite{He:81}) that $A_{\lambda}:=\lambda I - A_2$ is the generator of an analytic semigroup $\{e^{A_{\lambda}t}\}_{t\geq 0}$ in $H$ and the (classical) solution of \eqref{eq:AEp}, in the case $p=2$, is given by
$$
u_2(t) = e^{A_{\lambda}t}u_0 + \int_0^t e^{A_{\lambda}(t-s)}g\,ds, \quad t\geq 0.
$$
As in the previous Section, we can define the semigroup $\{T_2(t)\}_{t\geq 0} \subset \mathcal{C}(H;H)$ by
$$
T_2(t):u_0 \in  H \mapsto u_2(t) \in H, 
$$
the solution to \eqref{eq:AEp} (in the case $p=2$) starting at $u_0$ at the time $t$.

Since $\|e^{-A_2 t}\|_{\mathscr{L}(H)} \le e^{-\lambda_1 t}$, for all $t\geq 0$, where $\lambda_1$ denotes the first eigenvalue of the operator $A_2$, we have that $\|e^{A_{\lambda} t}\|_{\mathscr{L}(H)} \le e^{(\lambda-\lambda_1) t}$, for all $t\geq 0$. 
So, if $u_2^*$ is an equilibrium point of \eqref{eq:E2}, i.e., $u_2^* \in D(A_\lambda)$, and $A_\lambda u_2^*+g =0$, setting $w(t)=:u_2(t)-u_2^*$, we can easily see that $\dfrac{dw}{dt}= A_\lambda w$ and $w(t) = e^{A_{\lambda}t}w(0)$, for all $t\geq 0$.
Therefore
\begin{equation}\label{eq:3.1}
\|u(t)-u_2^*\|_{H} \le e^{(\lambda-\lambda_1) t}\|u_0 -u_2^*\|_{H}.
\end{equation}

On the other hand if $\{\varphi_j\}_{j=1}^\infty$ denotes the ortonormal basis of $H$ consisting of the (ordered) eigenvalues of the operator $A_2$, solution $u_2$ can be expressed as 
\begin{equation}\label{eq:3.1a}
u_2(t) = \sum_{j=1}^\infty \hat{u}_{2_j}(t)\,\varphi_j,
\end{equation}
where the Fourier modes $\hat{u}_{2_j}(t)$ are given by 
\begin{equation}\label{eq:fourmod}
\hat{u}_{2_j}(t) =e^{(\lambda-\lambda_j)t} \langle u_0, \varphi_j \rangle + \int_0^t e^{(\lambda-\lambda_j)(t-s)}ds\langle g, \varphi_j \rangle.
\end{equation}
It follows from Parseval's identity that 
\begin{equation}\label{eq:3.3}
\|u_2(t)\|^2_{H} \geq \begin{cases}
 	\displaystyle \left[ e^{(\lambda-\lambda_j)t} \langle u_0, \varphi_j \rangle + \left(\dfrac{e^{(\lambda-\lambda_j)t}-1}{\lambda-\lambda_j}\right) \langle g, \varphi_j \rangle  \right]^2, \quad \text{if }  \lambda \neq \lambda_j, \\
 	\displaystyle \left[\langle u_0, \varphi_j \rangle + t \langle g, \varphi_j \rangle  \right]^2, \quad \text{if }  \lambda = \lambda_j.  
 \end{cases}
\end{equation}
From this discussion we derive the following Proposition 
\begin{proposition}\label{prop:attrac2} \

\begin{enumerate}
	\item If $\lambda < \lambda_1$, the semigroup $\{T_2(t)\}_{t\geq 0}$ has a compact global attractor $\mathcal{A}_2=\{u^*_2\}$;
	
	\item If $\lambda \geq \lambda_1$ and $g \neq 0$, the semigroup $\{T_2(t)\}_{t\geq 0}$ admits unbounded orbits, i.e., there exist $u_0 \in H$ such that $\displaystyle \limsup_{t \to \infty} \|T_2(t)u_0\|_H=\infty$;
	
	\item If $\lambda = \lambda_1$ and $g = 0$, then for all $R>0$ there exists $u_0 \in H$ such that $\|T(t)u_0\|_H \geq R$.
 	\end{enumerate}
\end{proposition}
  
As a immediate consequence of Proposition \ref{prop:attrac2}, if $\lambda \geq \lambda_1$ there is no compact global attractor for the semigroup $\{T_2(t)\}_{t\geq 0}$. 
However, we consider a proper notion of non-compact global attractor suitable for slowly non-dissipative dynamical systems.

\begin{definition}\cite{PR:16}\label{def:unbattr}
	A non-compact global attractor for the semigroup $\{T_2(t)\}_{t\geq 0}$ is a non-empty minimal subset $\mathcal{A}_2$ of $H$ satisfying $T_2(t)\mathcal{A}_2 = \mathcal{A}_2$ for all $t\geq 0$ and $\displaystyle \lim_{t\to \infty} \, \sup_{b \in \mathcal{B}} \inf_{a\in \mathcal{A}_2} \|T_2(t)b - a\|_{H}=0$, for all bounded subset $\mathcal{B} \subset H$. 
\end{definition}
 
%

\begin{theorem}\label{teo:carac-att-unbou}
Let $\mathcal{A}_2 \subset H$ be a non-compact global attractor for the semigroup $\{T_2(t)\}_{t \geq 0}$. If $u_0 \in \mathcal{A}_2$, there exists a complete orbit for $\{T_2(t)\}_{t \geq 0}$ through $u_0$. 
Moreover, if there exists a complete orbit $\phi:\mathbb{R} \to H$ for $\{T_2(t)\}_{t \geq 0}$ through $u_0 \in H$ such that $\phi((-\infty,0]) \subset H$ is bounded, then $u_0 \in \mathcal{A}_2$.
\end{theorem}
\begin{proof}
The first part of the Theorem is proved by standard arguments based on the invariance of $\mathcal{A}_2$. 
For the reader's convenience we present it.

Let $u_0 \in \mathcal{A}_2$	be fixed. Since $\mathcal{A}_2=T_2(1)\mathcal{A}_2$, there exists $u_{-1} \in \mathcal{A}_2$ such that $T_2(1)u_{-1}=u_0$. 
By induction there exists a sequence $\{u_{-n}\}_{n\in \mathbb{N}}$ satisfying $T(1)u_{-n-1}=u_{-n}$ for all $n\in \mathbb{N}$. 

Finally, defining $\phi:\mathbb{R} \to H$ by
$$
\phi(t)=
\begin{cases}
T_2(t)u_0, \quad t\geq 0, \\
	     T_2(n+t)u_{-n}, \quad t \in [-n,-n+1), \, n=1,2, \cdots,
\end{cases}
$$
one obtains a complete orbit for $\{T_2(t)\}_{t \geq 0}$ through $u_0$.

On the other hand, let assume that $\phi:\mathbb{R} \to H$ is a complete orbit for $\{T_2(t)\}_{t \geq 0}$ through $u_0$ such that $\phi((-\infty,0]) \subset H$ is bounded.
Then
$$
\displaystyle \sup_{b \in \phi(\mathbb{R})} \inf_{a\in \mathcal{A}_2} \|b - a\|_{H} =  \lim_{t\to \infty} \, \sup_{b \in \phi((-\infty,t])} \inf_{a\in \mathcal{A}_2} \|b - a\|_{H} = \lim_{t\to \infty} \, \sup_{b \in \phi((-\infty,0])} \inf_{a\in \mathcal{A}_2} \|T_2(t)b - a\|_{H}=0,
$$ 	
and we conclude that $\phi(\mathbb{R}) \subset \mathcal{A}_2$.

\end{proof}

In the following we describe in some detail the structure of a such non-compact global attractor. 

If $\lambda_1 < \lambda_2 \le \lambda_3 \leq  \cdots$ denote the eigenvalues of the operator $A_2$ counted with multiplicity, let $N(\lambda)$ be the number of eigenvalues, counted with multiplicity, less or equal to $\lambda$, i.e., $N(\lambda)$ satisfies that $0< \lambda_1 < \lambda_2 \le \cdots \le \lambda_{N(\lambda)} \leq \lambda < \lambda_{N(\lambda)+1}$.

By \eqref{eq:fourmod} if one considers unbounded orbits, since its norm grows to infinity with time, the Fourier modes $\hat{u}_{2_j}(t)$ with $j > N(\lambda)$ will not affect the shape profile of such orbits as $t \to \infty$. 
Since we are concerned with its behavior at infinity, we are particularly concerned with the modes $\hat{u}_{2_j}(t)$ with $1\le j \leq  N(\lambda)$.

\begin{proposition}
	Consider an unbounded orbit $u_2(t)$ and its normalised trajectory $\dfrac{u_2(t)}{\|u_2(t)\|_H}$. Then $\dfrac{u_2(t)}{\|u_2(t)\|_H} \stackrel{t \to \infty}{\longrightarrow} \varphi_j$ in $H$, if and only if, $\dfrac{\hat{u}_{2_j}(t)}{\|u_2(t)\|_H} \stackrel{t \to \infty}{\longrightarrow} 1$.
\end{proposition}
\begin{proof}
Just note that
\begin{align*}
\left\|\dfrac{u_2(t)}{\|u_2(t)\|_H} - \varphi_j \right\|_H^2= 2 - 2 \langle \dfrac{u_2(t)}{\|u_2(t)\|_H} , \varphi_j \rangle	= 2-2 \dfrac{\hat{u}_{2_j}(t)}{\|u_2(t)\|_H}.
\end{align*}	
\end{proof}

\begin{corollary}
Consider an unbounded orbit $u_2(t)$ and its normalised trajectory $\dfrac{u_2(t)}{\|u_2(t)\|_H}$. 
Then $\hat{u}_{2_j}(t) \stackrel{t \to \infty}{\longrightarrow} \infty$ for some $1\le j \le  N(\lambda)$. 
Moreover $\dfrac{u_2(t)}{\|u_2(t)\|_H} \stackrel{t \to \infty}{\longrightarrow} \varphi_k$ in $H$, where $\displaystyle k=\min_{1\le j \le  N(\lambda)}\{j:\hat{u}_{2_j}(t) \stackrel{t \to \infty}{\longrightarrow} \infty\}$.
\end{corollary}
\begin{proof}
Its clear from \eqref{eq:3.3} that if $u_2(t)$ is an unbounded orbit then $\hat{u}_{2_j}(t) \stackrel{t \to \infty}{\longrightarrow} \infty$ for some $1\le j \leq N(\lambda)$.

Let $j \in (k, N(\lambda)]$ be fixed. 
Assuming that $\hat{u}_{2_j}(t) \stackrel{t \to \infty}{\longrightarrow} \infty$ its grows like $e^{(\lambda-\lambda_j)t}$. 
Since $e^{(\lambda-\lambda_j)t}< e^{(\lambda-\lambda_k)t}$ for all $t$ large enough, we have that
$$
\lim_{t\to \infty} \dfrac{\hat{u}_{2_j}(t)}{\|u_2(t)\|_H} = \lim_{t\to \infty} \dfrac{\hat{u}_{2_j}(t)}{\hat{u}_{2_k}(t)} = \lim_{t\to \infty} e^{(\lambda_k - \lambda_j)t}=0.
$$	
\end{proof}

\begin{remark}
In this simple case ($p=2$) where we have explicit solutions for the problem \eqref{eq:AEp}, we can explicitly describe the non-compact global attractor $\mathcal{A}_2$.
 
Considering the eigenfunctions $\varphi_j$, $j=1, 2, \cdots, N(\lambda)$ as initial data, it is not difficult to see that all the asymptotic behavior of $u_2(t)$ in the infinite is captured by the subspace generated by $\varphi_j$, with $1 \leq j\leq N(\lambda)$, which is minimal with such property, while bounded solutions must converge to $u^*_2$ exponentially fast. 
Therefore the non-compact global attractor $\mathcal{A}_2$ can be decomposed as
$$
\mathcal{A}_2 = \mathcal{A}_2^u \cup \{u^*_2\}, 
$$
where $\mathcal{A}_2^u:={\rm span} \{\varphi_1, \varphi_2, \cdots, \varphi_{N(\lambda)}\}$ is the unbounded part of the global attractor. 

In the following we describe the approach developed in \cite{He:13}, based on Poincar\'e's projection, which allow us to interpret $\varphi_j$, $1 \leq j\leq N(\lambda)$ as equilibrium points of \eqref{eq:E2} at infinity, and therefore, unbounded orbits may be interpreted as heteroclinic connections to infinity.

First we embed $H$ into $H \times \mathbb{R}$ and we identify $H$ with the affine hyperplane $H \times \{1\}$ through the bijection $H \ni v \mapsto (v,1) \in H \times \{1\}$.  
Finally, the affine hyperplane $H \times \{1\}$ is projected stereographically to the Poincar\'e hemisphere $\mathcal{S}:= \{(v,s) \in H \times \mathbb{R}: \|v\|_H^2 + s^2=1, \, s \geq 0\}$. This transformation is given explicitly by the formula
$$
\mathcal{P}:v \in H \mapsto \dfrac{(v,1)}{\|v\|_H^2 + 1} \in \mathcal{S}
$$ 

\begin{center}
\begin{tikzpicture}\usetikzlibrary{arrows,shapes,positioning}\usetikzlibrary{decorations.markings} \tikzstyle arrowstyle=[scale=1]
\tikzstyle directed=[postaction={decorate,decoration={markings, mark=at position .65 with {\arrow[arrowstyle]{stealth}}}}]
\tikzstyle reverse directed=[postaction={decorate, decoration={markings, mark=at position .65 with {\arrowreversed[arrowstyle]{stealth};}}}]
\coordinate (O) at (-0.5,0) ; \coordinate (A) at (-0.5,4) ; \coordinate (B) at (-0.5,-2) ; \coordinate (C) at (-5,0) ; \coordinate (D) at (5,0) ;
\coordinate (E) at (-5,2) ; \coordinate (F) at (3,2) ;
\coordinate (G) at (-3,3) ; \coordinate (H) at (5,3) ;
\coordinate (P) at (3,2.5) ; 
\draw(A) -- (B) ; \node[] at (94:3.8) {$\mathbb{R}$};
\draw (C) -- (D) ; \node[] at (2:4.8) {\footnotesize{$H$}};
\draw[thick, directed] (P) -- (O); \node[] at (40:4.1) {$v$}; \node[] at (35:2.3) {$\mathcal{P}v$}; \node[] at (147:4.2) {\footnotesize{$H\times \{1\}$}};
\draw (E) -- (F) ;
\draw[dash pattern=on5pt off3pt] (E) -- (G) ;
\draw[dash pattern=on5pt off3pt] (G) -- (H) ;
\draw[dash pattern=on5pt off3pt] (H) -- (F) ;
\draw[dash pattern=on5pt off3pt] (2,0) arc (0:180:2.4);
\draw (2,0) arc (0:53.8:2.4);
\draw (-2.8,0) arc (180:125:2.4);
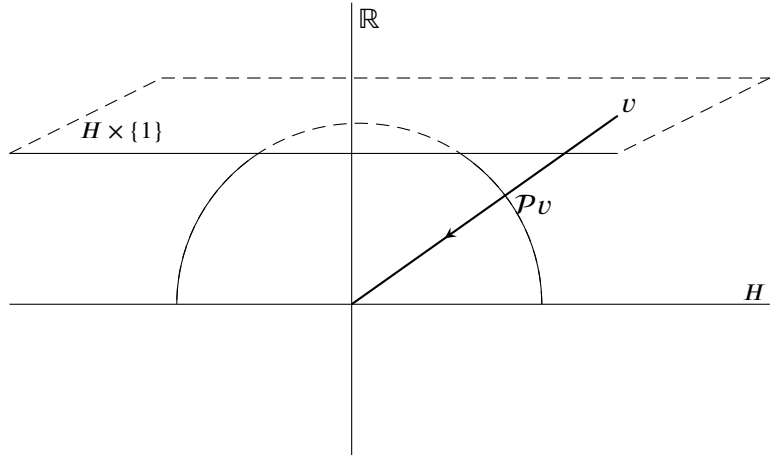
\node[] at (-96:2.4) {\text{The Poincar\'e's projection}};
		
\end{tikzpicture}
\end{center}                                                                                   

As $\|v\|_H$ goes to infinity, $\mathcal{P}v$ goes to the equator $\mathcal{E}:=\{(v,0)\in H \times \mathbb{R}: \|v\|_H^2=1\}$ of the Poincar\'e hemisphere. 
By applying the Poincar\'e transformation at equation \eqref{eq:AEp} (in the case $p=2$) was observed in \cite{He:13} that $\mathcal{P}^{-1}(\varphi_j,0)$, $1 \leq j\leq N(\lambda)$ play the role of equilibria for \eqref{eq:E2} at infinity. 
In addition to that, this Poincar\'e ``compactification'' of the phase space allows a rigorous analysis of the dynamics at infinity in more complicated situations, e.g. nonlinear equations.
\end{remark}

\section{Continuity properties of flows and attractors}\label{sec:cont}

We start with the continuity of the (bounded) equilibrium points.

\begin{theorem}\label{theo:4.1}
Let $u_p^*$ and $u_2^*$	equilibrium points of \eqref{eq:Ep} and \eqref{eq:E2} respectively. Then $\displaystyle \|u^*_p - u^*_2\|_V \stackrel{p\to 2^+}{\longrightarrow} 0$.
\end{theorem}

This Theorem is proved in \cite[Theorem 1.3 and Remark 4.2]{Li:87} for the equation $-{\rm div}(|\nabla u|^{p-2} \nabla u)=g$ as $p\to 2$. 
The proof is based on variational arguments on the energy functional $\displaystyle \mathcal{E}_p(u)= \dfrac1p \int_{\Omega} |\nabla u|^p \, dx - \int_{\Omega} g\, u \,dx$. On letting $\displaystyle \tilde{\mathcal{E}_p}(u)= \dfrac1p \int_{\Omega} |\nabla u|^p \, dx - \dfrac{\lambda}{2} \int_{\Omega} u^2 \, dx - \int_{\Omega} g\, u \,dx$ we get the  result. 

\

In the following we address to the continuity of semigroups.

\begin{theorem}\label{theo:4.2}
Given $u_{p}^0,\, u_{2}^0 \in H$, let $T_p(t)u_{p}^0$ and $T_2(t)u_{2}^0$ be the semigroups defined in Section \ref{sec:pnot2} and Section \ref{sec:p2} respectively. 
There exist functions $C_1(p,R,T)$ with $C_1(p,R,T) \stackrel{p \to 2}{\longrightarrow} 0$, and $C_2(T)$ such that
\begin{equation}\label{eq:4.1}
 \|T_p(t)u_p^0 - T_2(t)u_2^0\|_H \le C_1(p,R,T) + C_2(T) \|u_{p}^0 - u_{2}^0\|_H,
\end{equation}
for $t \in [0,T]$ and $\|u_p^0\|_H,\, \|u_2^0\|_H \le R$.
\end{theorem}
\begin{proof}
Let $u_p(t):= T_p(t)u_p^0$ and $u_2(t):=T_2(t)u_2^0$ and set $w(t) := u_p(t)- u_2(t)$. Then $\dfrac{dw}{dt} = A_2 u_2(t) - A_p u_p(t) + \lambda w$. 
From Tartar's inequality we obtain that
\begin{align*}
\langle A_p u_p - A_2 u_2, w \rangle & =  \langle  |\nabla u_p|^{p-2} \nabla u_p - |\nabla u_2|^{p-2} \nabla u_2, \nabla w \rangle + \langle|\nabla u_2|^{p-2} \nabla u_2- \nabla u_2, \nabla w \rangle \\
& \geq 2^{2-p} \| \nabla w \|_{H}^p + \langle|\nabla u_2|^{p-2} \nabla u_2- \nabla u_2, \nabla w \rangle. 
\end{align*}
So,
\begin{align*}
\frac12 \frac{d}{dt} \| w \|^2_{H} & \leq  - \langle|\nabla u_2|^{p-2} \nabla u_2- \nabla u_2, \nabla w \rangle + \lambda  \| w_p \|^2_{H} \\
& \le  \int_{\Omega}\left| |\nabla u_2|^{p-1} - |\nabla u_2| \right| \,|\nabla w|\, dx + \lambda \| w_p \|^2_{H}. 
\end{align*}

In order to estimate $\left| |\nabla u_2|^{p-1} - |\nabla u_2| \right|$, we observe that from the mean value Theorem one gets
$$
\left| |\nabla u_2|^{p-1} - |\nabla u_2| \right| = |\nabla u_2|^{r} \left| {\rm ln}|\nabla u_2| \right||p-2|,
$$
for some $r \in (1,p-1)$ as long as $|\nabla u_2| \neq 0$. 
It follows from uniform estimate in Theorem \ref{teo:CCD} and \eqref{eq:fourmod} that $|\nabla u_2 (t)|, |\nabla u_p(t)|$ are uniform bounded for $\|u_p^0\|_H,\, \|u_2^0\|_H \le R$ and $t \in [0,T]$.

Therefore
$$
\frac12 \frac{d}{dt} \| w (t)\|^2_{H} \le C|p-2| + \lambda \|w(t)\|^2_H, \quad t \in [0,T],
$$
Integrating this last inequality from $0$ to $t$ we obtain
$$
\| w (t)\|^2_{H} \le 2 C t|p-2| + \|u_p^0- u_2^0\|_H  +  2\lambda \int_0^t\|w(s)\|^2_H \, ds.
$$
By Gronwall's Inequality we conclude that
$$
\| w(t)\|^2_{H} \le C_1|p-2|+ C_2\|u_p^0- u_2^0\|_H ,
$$
for $\|u_p^0\|, \, \|u_2^0\|_H \le R$ and $t \in [0,T]$. 

\end{proof}
\subsection{Continuity of attractors}\label{sec:cont-attr}

The Hausdorff semi-distance of two subsets $\mathcal{A}, \mathcal{B}$ of a metric space $(X,d)$ is defined as
 $$
{\rm dist}_{\mathcal{H}}(\mathcal{A},\mathcal{B}):=\sup_{a\in \mathcal{A}} \inf_{b\in \mathcal{B}} d(a,b).
$$
Given a family $\{\mathcal{A_\lambda}\}_{\lambda \in \Lambda}$ of subsets of $X$, we say that $\{A_\lambda\}_{\lambda \in \Lambda}$ is upper semicontinuous at $\lambda=\lambda_0 \in \Lambda$ if,
$$
{\rm dist}_{\mathcal{H}}({\mathcal{A}}_\lambda,{\mathcal{A}}_{\lambda_0}) \stackrel{\lambda \to \lambda_0}{\longrightarrow} 0.
$$
We say that $\{A_\lambda\}_{\lambda \in \Lambda}$ is lower semicontinuous at $\lambda=\lambda_0 \in \Lambda$ if,
$$
{\rm dist}_{\mathcal{H}}({\mathcal{A}}_{\lambda_0},{\mathcal{A}}_{\lambda}) \stackrel{\lambda \to \lambda_0}{\longrightarrow} 0.
$$
We say that $\{A_\lambda\}_{\lambda \in \Lambda}$ is continuous at $\lambda=\lambda_0 \in \Lambda$ if is both upper and lower semicontinuous at $\lambda=\lambda_0$.

\begin{theorem}\cite{Ha:88} 
\begin{enumerate}
	\item  A family $\{\mathcal{A_\lambda}\}_{\lambda \in \Lambda}$ of subsets of $X$ is upper semicontinuous at $\lambda = \lambda_0 \in \Lambda$, if and only if, for any sequences $\lambda_n \stackrel{n \to \infty}{\longrightarrow} \lambda_0$ and $u_{\lambda_n} \in \mathcal{A}_{\lambda_n}$ there exists a subsequence of $\{\lambda_n\}$ (still denoted by $\{\lambda_n\}$) such that $u_{\lambda_n} \stackrel{n \to \infty}{\longrightarrow} u_{\lambda_0}$ for some $u_{\lambda_0} \in \mathcal{A}_{\lambda_0}$.
	
	\item Assuming $\mathcal{A}_{\lambda_0}$ a compact subset of $X$, a family $\{\mathcal{A_\lambda}\}_{\lambda \in \Lambda}$ of subsets of $X$ is lower semicontinuous at $\lambda = \lambda_0 \in \Lambda$, if and only if, for any sequence $\lambda_n \stackrel{n \to \infty}{\longrightarrow} \lambda_0$ and any $u_{\lambda_0} \in \mathcal{A}_{\lambda_0}$ there exists a subsequence of $\{\lambda_n\}$ (still denoted by $\{\lambda_n\}$) and a sequence $u_{\lambda_n} \in \mathcal{A}_{\lambda_n}$, such that $u_{\lambda_n} \stackrel{n \to \infty}{\longrightarrow} u_{\lambda_0}$.
\end{enumerate}	
\end{theorem}

\begin{theorem} \
\begin{enumerate}
	\item If $\lambda < \lambda_1$, the family of compact global attractors $\{\mathcal{A}_p\}_{p\geq 2} \subset H$ is continuous at $p=2$;
	
	\item If $\lambda_1 \le \lambda $ the family of global attractors $\{\mathcal{A}_p\}_{p\geq 2} \subset H$ is upper semicontinuous at $p=2$.
	\end{enumerate}
\end{theorem}
\begin{proof} \
\begin{enumerate}
	\item This is consequence of Theorem \ref{theo:4.1}.
	
	\item  Let us consider sequences $p_n \stackrel{n \to \infty}{\longrightarrow} 2^+$ and $u_{p_n} \in \mathcal{A}_{p_n}$, $n \in \mathbb{N}$. 
By the a priori estimate in Theorem \ref{teo:CCD} it follows that $\displaystyle \overline{ \bigcup_{2<p\le 3}\mathcal{A}_p}$ is a compact subset of $H$, therefore (passing to a subsequence if necessary) $u_{p_n} \stackrel{n \to \infty}{\longrightarrow} u_2$ for some $u_2 \in H$. It remains to show that $u_2 \in \mathcal{A}_2$. 
For this, it is enough to prove the existence of a complete orbit for $\{T_2(t)\}_{t\geq 0}$ through $u_2$ bounded in the past.

For each $n \in \mathbb{N}$, there exists a bounded complete orbit $\phi_{p_n}:\mathbb{R} \to H$ for $\{T_{p_n}(t)\}_{t\geq 0}$ through $u_{p_n}$.
For $t \geq 0$, it follows from the continuity of the semigroups in Theorem \ref{theo:4.2} that $\phi_{p_n}(t)=T_{p_n}(t)u_{p_n} \stackrel{n \to \infty}{\longrightarrow} T_{2}(t)u_{2}$ in $H$.

For $t<0$, we set up the orbit through $u_{2}$ in the following way: for $t \in (-k,-k+1]$, $k \in \mathbb{N}$, we consider the sequence $\{\phi_{p_n}(-k)\}_{n \in \mathbb{N}}$ in $\displaystyle \bigcup_{2 <p \le 3}\mathcal{A}_p$. 
Following the previously argument used for the sequence $\{u_{p_n}\}_{n\in \mathbb{N}}$, we obtain that $\phi_{p_n}(-k) \stackrel{n \to \infty}{\longrightarrow} \tilde\phi_{2}(-k)$ in $H$. 
Hence,
$$
\phi_{p_n}(t) = T_{p_n}(t+k)\phi_{p_n}(-k) \stackrel{n \to \infty}{\longrightarrow}T_{2}(t+k)\tilde\phi_{2}(-k).
$$
Finally, defining $\phi_{2}: \mathbb{R} \to H$ by
$$
\phi_{2}(t)= \begin{cases}
	T_2(t)u_2, \quad t\geq 0, \\
	T_{2}(t+k)\tilde\phi_{2}(-k), \quad t \in [-k,-k+1), \, k=1,2, \cdots, 
			\end{cases}
$$
one obtains a complete orbit for $\{T_{2}(t)\}_{t\geq 0}$ through $u_2$ such that $\phi_{2}((-\infty,0]) \subset H$ is bounded. The result follows now from Theorem \ref{teo:carac-att-unbou}.
\end{enumerate}	
\end{proof}

\begin{remark}
To address to the lower semicontinuity of the family of compact global attractors for  dissipative systems possessing a Lyapunov function (gradient systems), we can indicate the general scheme developed in \cite{AP:06, HR:89}.  
Such scheme is often applied in the case of semilinear equations (see \cite{CLR:12}) by assuming some robustness on the structure of the equilibria set, e.g. hyperbo\-li\-ci\-ty. 
However, for quasilinear equations still it is an open problem how to apply it. We hope to address to this issue in a further paper. 	
\end{remark}
\begin{appendix}
\section{Appendix}	

For reader's convenience we collect some of the no so standard inequalities that we had use in the paper.

\begin{lemma}[Tartar's Inequality, \cite{Sa:85}, Lemma 3.1]\label{lem:tartar}
For all $\xi, \eta \in \mathbb{R}^n$ and $p\geq 2$ the following inequality hold
$$
2^{2-p} |\xi - \eta|^p \le (|\xi|^{p-2}\xi - |\eta|^{p-2}\eta) \cdot (\xi - \eta).
$$	
\end{lemma}

\begin{lemma}[Ghidaglia's Inequality, \cite{Te:88}, III-Lemma 5.1]\label{lem:ghidaglia} Let $y:(0,\infty) \to \mathbb{R}$ be a positive absolutely continuous function which satisfies
$$
\dfrac{dy}{dt} + \gamma y^{p/2} \le \delta, 
$$
with $p>2$, $\gamma >0$ and $\delta \geq 0$. Then 
$$
y(t) \le \left(\dfrac{\delta}{\gamma} \right)^{2/p} + \left(\dfrac{\gamma(p-2)}{2}t \right)^{-2/(p-2)}, \quad \forall \, t>0.
$$	
\end{lemma}

\end{appendix}

\textbf{Acknowledgements:}

 R.P.S. is partially supported by FAPDF $\#193.001.372/2016$.


\end{document}